\documentclass[12pt]{amsproc}
\usepackage{hyperref}
\usepackage{geometry}
\usepackage{graphicx}
\usepackage{amsthm}
\usepackage{amssymb}

\theoremstyle{plain}
\newtheorem{thm}{Theorem}[section]
\newtheorem{cor}[thm]{Corollary}
\newtheorem{lem}[thm]{Lemma}
\newtheorem{prop}[thm]{Proposition}
\newtheorem{defn}[thm]{Definition}
\newtheorem{exa}[thm]{Example}
\newtheorem{rem}[thm]{Remark}

\begin{document}

\title{Graded Classical Weakly Prime Submodules over Non-Commutative Graded Rings}
\author{Jebrel M. \textsc{Habeb$^{*}$}}\thanks{This work had been done while the fist author has a sabbatical leave from Yarmouk University for the academic year  2022-2023. }\thanks {$^{*}$Corresponding Author: Jebrel M. Habeb: e-mail jhabeb@yu.edu.jo}
\address{Department of Mathematics, Yarmouk University, Irbid, Jordan}
\email{jhabeb@yu.edu.jo}

\author{Rashid \textsc{Abu-Dawwas}}
\address{Department of Mathematics, Yarmouk University, Irbid, Jordan}
\email{rrashid@yu.edu.jo}



\begin{abstract}
The goal of this article is to propose and examine the notion of graded classical weakly prime submodules over non-commutative graded rings which is a generalization of the concept of graded classical weakly prime submodules over commutative graded rings. We investigate the structure of these types of submodules in various categories of graded modules. 
\end{abstract}

\maketitle

\section{Introduction}

During the whole of this article, all our rings are associative with nonzero unity $1$ and all modules are unital left modules. Let $G$ be a multiplicative group with identity $e$ and $A$ be a ring with nonzero unity 1. Then $A$ is called $G$-graded if $A=\displaystyle\bigoplus_{g\in G} A_{g}$ with $A_{g}A_{h}\subseteq A_{gh}$ for all $g, h\in G$, where $A_{g}$ is an additive subgroup of $A$ for all $g\in G$, here $A_{g}A_{h}$ denotes the additive subgroup of $A$ consisting of all finite sums of elements $a_{g}b_{h}$ with $a_{g}\in A_{g}$ and $b_{h}\in A_{h}$. We denote this by $G(A)$. The elements of $A_{g}$ are called homogeneous of degree $g$. If $a\in A$, then $a$ can be written uniquely as $a=\displaystyle\sum_{g\in G}a_{g}$, where $a_{g}$ is the component of $a$ in $A_{g}$ and $a_{g}=0$ except for finitely many. The additive subgroup $A_{e}$ is in fact a subring of $A$ and $1\in A_{e}$. The set of all homogeneous elements of $A$ is $\displaystyle\bigcup_{g\in G}A_{g}$ and is denoted by $h(A)$. Let $P$ be a left ideal of a $G$-graded ring $A$. Then $P$ is called a graded left ideal if $P=\displaystyle\bigoplus_{g\in G}(P\cap A_{g})$, i.e., for $a\in P$, $a=\displaystyle\sum_{g\in G}a_{g}$ where $a_{g}\in P$ for all $g\in G$. A left ideal of a graded ring is not necessarily a graded left ideal, see (\cite{Alshehry Dawwas}, Example 1.1).

Let $A$ be a $G$-graded ring. A left $A$-module $M$ is called $G$-graded if $M=\displaystyle\bigoplus_{g\in G} M_{g}$ with $A_{g}M_{h}\subseteq M_{gh}$ for all $g, h\in G$, where $M_{g}$ is an additive subgroup of $M$ for all $g\in G$. The elements of $M_{g}$ are called homogeneous of degree $g$. If $m\in M$, then $m$ can be written uniquely as $m=\displaystyle\sum_{g\in G}m_{g}$, where $m_{g}$ is the component of $m$ in $M_{g}$ and $m_{g}=0$ except for finitely many. Certainly, $M_{g}$ is an $A_{e}$-module for all $g\in G$. The set of all homogeneous elements of $M$ is $\displaystyle\bigcup_{g\in G}M_{g}$ and is denoted by $h(M)$. Let $K$ be an $A$-submodule of $M$. Then $K$ is called a graded submodule of $M$ if $K=\displaystyle\bigoplus_{g\in G}(K\cap M_{g})$, i.e., for $x\in K$, $x=\displaystyle\sum_{g\in G}x_{g}$ where $x_{g}\in K$ for all $g\in G$. Similar to the case of graded ideals, an $A$-submodule of a graded $A$-module is not necessarily a graded submodule. For more details and terminologies, see \cite{Hazart, Nastasescue}.

Graded prime ideals of commutative graded rings have been introduced and studied in \cite{Refai Hailat Obiedat}. A proper graded ideal $P$ of a commutative graded ring $A$ is said to be graded prime if whenever $x, y\in h(A)$ such that $xy\in P$, then either $x\in P$ or $y\in P$. The concept of graded prime ideals and its generalizations have an outstanding location in graded commutative algebra. They are valuable tools to determine the properties of graded commutative rings. Numerous generalizations of graded prime ideals have been investigated. Indeed, Atani introduced in \cite{Atani} the concept of graded weakly prime ideals. A proper graded ideal $P$ of a commutative graded ring $A$ is said to be a graded weakly prime ideal if whenever $x, y\in h(A)$ such that $0\neq xy\in P$, then $x\in P$ or $y\in P$. By (\cite{Atani}, Theorem 2.12), the following statements are equivalent for a graded ideal $P$ of $G(A)$ with $P\neq A$, where $A$ is a commutative graded ring:

\begin{enumerate}
\item $P$ is a graded weakly prime ideal of $G(A)$.

\item For each $g, h\in G$, the inclusion $0\neq IJ\subseteq P$ with $A_{e}$-submodules $I$ of $A_{g}$ and $J$ of $A_{h}$ implies that $I\subseteq P$ or $J\subseteq P$.
\end{enumerate}

For graded rings that are not necessarily commutative, it is clear that (2) does not imply (1). So, Alshehry and Abu-Dawwas in \cite{Alshehry Dawwas} defined a graded left ideal of $G(A)$, where $A$ need not be commutative to be a graded weakly prime as follows: a graded left ideal $P$ of $G(A)$ with $P\neq A$ is said to be a graded weakly prime ideal of $G(A)$ if for each $g, h\in G$, the inclusion $0\neq IJ\subseteq P$ with $A_{e}$-submodules $I$ of $A_{g}$ and $J$ of $A_{h}$ implies that $I\subseteq P$ or $J\subseteq P$. Equivalently, whenever $x, y\in h(R)$ such that $0\neq xRy\subseteq P$, then either $x\in P$ or $y\in P$ (\cite{Alshehry Dawwas}, Proposition 2.3). In \cite{Dawwas Bataineh Muanger}, the standard definition of a graded prime ideal $P$ for a graded non-commutative ring $A$ is that $P\neq A$ and whenever $I$ and $J$ are graded left ideals of $A$ such that $IJ\subseteq P$, then either $I\subseteq P$ or $J\subseteq P$. Accordingly, in \cite{Alshehry Dawwas}, they defined a graded left ideal of a graded ring $A$ to be a graded weakly prime as follows: a proper graded left ideal $P$ of $A$ is said to be a graded weakly prime ideal of $A$ if whenever $I$ and $J$ are graded left ideals of $A$ such that $0\neq IJ\subseteq P$, then either $I\subseteq P$ or $J\subseteq P$.

Let $K$ be a graded $A$-submodule of a left $A$-module $M$. Then $(K:_{A}M)=\left\{a\in A:aM\subseteq K\right\}$ is a graded two-sided ideal of $A$ \cite{Farzalipour}. $K$ is said to be faithful if $Ann_{A}(K)=(0:_{A}K)=0$. In \cite{Dawwas Bataineh Muanger}, a proper graded $A$-submodule $K$ of a graded $A$-module $M$ over non-commutative graded ring $A$ is said to be graded prime if whenever $L$ is a graded $A$-submodule of $M$ and $I$ is a graded ideal of $A$ such that $IL\subseteq K$, then either $L\subseteq K$ or $I\subseteq(K:_{A}M)$. If $A$ is commutative, this definition is equivalent to: a proper graded $A$-submodule $K$ of a graded $A$-module $M$ is said to be graded prime if whenever $a\in h(A)$ and $x\in h(M)$ are such that $ax\in K$, then either $x\in K$ or $a\in(K :_{A} M)$. In a similar way, a proper graded $A$-submodule $K$ of a graded $A$-module $M$ over a non-commutative graded ring $A$ is said to be graded weakly prime if whenever $L$ is a graded $A$-submodule of $M$ and $I$ is a graded ideal of $A$ such that $0\neq IL\subseteq K$, then either $L\subseteq K$ or $I\subseteq(K:_{A}M)$. If $A$ is commutative, this definition is equivalent to: a proper graded $A$-submodule $K$ of a graded $A$-module $M$ is said to be graded weakly prime if whenever $a\in h(A)$ and $x\in h(M)$ are such that $0\neq ax\in K$, then either $x\in K$ or $a\in(K :_{A} M)$. Let $K$ be a graded $A$-submodule of $M$ and $g\in G$ such that $K_{g}\neq M_{g}$. Then $K$ is said to be a $g$-prime $A$-submodule of $M$ if whenever $L$ is an $A_{e}$-submodule of $M_{g}$ and $I$ is an ideal of $A_{e}$ such that $IL\subseteq K$, then either $L\subseteq K$ or $I\subseteq(K:_{A}M)$. 

The concept of graded weakly classical prime submodules over commutative graded rings has been proposed and studied by Abu-Dawwas and Al-Zoubi in \cite{Dawwas Zoubi}. A proper graded $A$-submodule $K$ of $M$ over a commutative graded ring $A$ is said to be a graded weakly classical prime if whenever $x, y\in h(A)$ and $m\in h(M)$ such that $0\neq xym\in K$, then either $xm\in K$ or $ym\in K$. In this article, we introduce and examine the concept of graded classical weakly prime submodules over non-commutative graded rings. Indeed, this article is motivated by the concepts and the techniques that have been examined in \cite{Jamali}. We propose the following: a proper graded $A$-submodule $K$ of $M$ over a non-commutative graded ring $A$ is said to be a graded classical weakly prime if whenever $x, y\in h(A)$ and $L$ is a graded $A$-submodule of $M$ such that $0\neq xAyL\subseteq K$, then either $xL\subseteq K$ or $yL\subseteq K$. Several properties have been examined. Also, we investigate the structure of graded classical weakly prime submodules in various categories of graded modules.

\section{Graded Classical Weakly Prime Submodules}

In this section, we introduce and examine graded classical weakly prime submodules over non-commutative graded rings.

\begin{defn}
Let $A$ be a graded ring, $M$ be a graded $A$-module, $K$ be a graded $A$-submodule of $M$ and $g\in G$. Then 
\begin{enumerate}
    \item $K$ is said to be a graded classical weakly prime $A$-submodule of $M$ if $K\neq M$ and whenever $x, y\in h(A)$ and $L$ is a graded $A$-submodule of $M$ such that $0\neq xAyL\subseteq K$, then either $xL\subseteq K$ or $yL\subseteq K$.
    \item $K$ is said to be a graded completely classical weakly prime $A$-submodule of $M$ if $K\neq M$ and whenever $x, y\in h(A)$ and $z\in h(M)$ such that $0\neq xyz\in K$, then either $xz\in K$ or $yz\in K$.
    \item $K$ is said to be a $g$-classical weakly prime $A$-submodule of $M$ if $K_{g}\neq M_{g}$ and whenever $x, y\in A_{e}$ and $L$ is an $A_{e}$-submodule of $M_{g}$ such that $0\neq xA_{e}yL\subseteq K$, then either $xL\subseteq K$ or $yL\subseteq K$.
\end{enumerate}
\end{defn}

Evidently, if $A$ is a commutative graded ring with unity, then the concepts of graded classical weakly prime submodules and graded completely classical weakly prime submodules coincide. Indeed, the next example demonstrates that this will not be the case for non-commutative graded rings:

\begin{exa}\label{1}
Consider $A=M=M_{2}(\mathbb{Z})$ (The $2\times2$ matrices over the ring of integers $\mathbb{Z}$) and $G=\mathbb{Z}_{4}$ (The additive group of integers modulo 4). Then $A$ is $G$-graded by $A_{0}=\left(\begin{array}{cc}
        \mathbb{Z} & 0 \\
        0 & \mathbb{Z} 
      \end{array}\right)$, $A_{2}=\left(\begin{array}{cc}
        0 & \mathbb{Z} \\
        \mathbb{Z} & 0 
      \end{array}\right)$ and $A_{1}=A_{3}=\left(\begin{array}{cc}
        0 & 0 \\
        0 & 0 
      \end{array}\right)$. $M$ also is a $G$-graded left $A$-module by the same graduation of $A$. Then $K=M_{2}(2\mathbb{Z})$ is a graded prime $A$-submodule of $M$, and hence $K$ is a graded classical weakly prime $A$-submodule of $M$. On the other hand, $K$ is not a graded completely classical weakly prime $A$-submodule of $M$ since $x=\left(\begin{array}{cc}
        1 & 0 \\
        0 & 2 
      \end{array}\right)\in h(A)$, $y=\left(\begin{array}{cc}
        0 & 1 \\
        1 & 0 
      \end{array}\right)\in h(A)$ and $z=\left(\begin{array}{cc}
        1 & 0 \\
        0 & 4 
      \end{array}\right)\in h(M)$ such that $0\neq xyz\in K$, $xz\notin K$ and $yz\notin K$.
\end{exa}

Clearly, the zero submodule is always graded classical weakly prime and $g$-classical weakly prime for all $g\in G$ by definition. The next example shows that one can find $g\in G$ such that the zero submodule is not a $g$-prime. 

\begin{exa}
Consider $A=M=M_{2}(\mathbb{Z})$ and $G=\mathbb{Z}_{4}$. Then $A$ is $G$-graded by $A_{0}=\left(\begin{array}{cc}
        \mathbb{Z} & 0 \\
        0 & \mathbb{Z} 
      \end{array}\right)$, $A_{2}=\left(\begin{array}{cc}
        0 & \mathbb{Z} \\
        \mathbb{Z} & 0 
      \end{array}\right)$ and $A_{1}=A_{3}=\left(\begin{array}{cc}
        0 & 0 \\
        0 & 0 
      \end{array}\right)$. $M$ also is a $G$-graded left $A$-module by the same graduation of $A$. Choose $x=\left(\begin{array}{cc}
        0 & 0 \\
        0 & 1 
      \end{array}\right)\in A_{0}$. Then $I=A_{0}xA_{0}=\left(\begin{array}{cc}
        0 & 0 \\
        0 & \mathbb{Z} 
      \end{array}\right)$ is an ideal of $A_{0}$. Choose $y=\left(\begin{array}{cc}
        0 & 1 \\
        0 & 0 
      \end{array}\right)\in M_{2}$. Then $L=A_{0}y=\left(\begin{array}{cc}
        0 & \mathbb{Z} \\
        0 & 0 
      \end{array}\right)$ is an $A_{0}$-submodule of $M_{2}$. Consider the graded $A$-submodule $K=\left(\begin{array}{cc}
        0 & 0 \\
        0 & 0 
      \end{array}\right)$ of $M$. Then $K_{2}=\left(\begin{array}{cc}
        0 & 0 \\
        0 & 0 
      \end{array}\right)\neq M_{2}$ and $IL\subseteq K$, but $L\nsubseteq K$ and $I\nsubseteq (K:_{A}M)$. Hence, $K$ is not a $2$-prime $A$-submodule of $M$. 
\end{exa}

\begin{thm}\label{Theorem 2.1}
Let $M$ be a graded $A$-module, $K$ be a graded $A$-submodule of $M$ and $g\in G$ such that $K_{g}\neq M_{g}$ and every $A_{e}$-submodule of $M_{g}$ is faithful. Then $K$ is a $g$-classical weakly prime $A$-submodule of $M$ if and only if whenever $I, J$ are ideals of $A_{e}$ and $L$ is an $A_{e}$-submodule of $M_{g}$ with $0\neq IJL\subseteq K$, then either $IL\subseteq K$ or $JL\subseteq K$. 
\end{thm}

\begin{proof}
Suppose that $K$ is a $g$-classical weakly prime $A$-submodule of $M$. Assume that $I, J$ are ideals of $A_{e}$ and $L$ is an $A_{e}$-submodule of $M_{g}$ with $0\neq IJL\subseteq K$. Suppose that $IL\nsubseteq K$ and $JL\nsubseteq K$. Then there exist $r\in I$ and $s\in J$ such that $rL\nsubseteq K$ and $sL\nsubseteq K$, and then $rA_{e}sL\subseteq IJL\subseteq K$, which implies that $rA_{e}sL=0$ as $K$ is $g$-weakly classical prime. So, $rsL=0$, and hence $rs=0$. Indeed, we show that $IJ=0$. Let $a\in I$ and $b\in J$. If $aL\nsubseteq K$ and $bL\nsubseteq K$, then $ab=0$ as above. If $aL\nsubseteq K$ and $bL\subseteq K$, then $(s+b)L\nsubseteq K$, and then $a(s+b)=0$. Since $as=0$, $ab=0$. Similarly, if $aL\subseteq K$ and $bL\nsubseteq K$, then $ab=0$. If $aL\subseteq K$ and $bL\subseteq K$, then $(r+a)L\nsubseteq K$ and $(s+b)L\nsubseteq K$, and then $(r+a)s=r(s+b)=(r+a)(s+b)=rs=0$, which gives that $ab=0$. Thus $IJ=0$, and then $IJL=0$, which is a contradiction. Hence, either $IL\subseteq K$ or $JL\subseteq K$. Conversely, let $x, y\in A_{e}$ and $L$ be an $A_{e}$-submodule $M_{g}$ with $0\neq xA_{e}yL\subseteq K$. Then $I=A_{e}x$ and $J=A_{e}y$ are ideals of $A_{e}$ with $0\neq IJL\subseteq K$, and then either $IL\subseteq K$ or $JL\subseteq K$ by assumption, and hence either $xL\subseteq K$ or $yL\subseteq K$. Thus $K$ is a $g$-classical weakly prime $A$-submodule of $M$.
\end{proof}

\begin{cor}\label{Corollary 2.2}
Let $M$ be a graded $A$-module, $g\in G$ such that every $A_{e}$-submodule of $M_{g}$ is faithful, and $K$ be a $g$-classical weakly prime $A$-submodule of $M$. Suppose that $L$ is an $A_{e}$-submodule of $M_{g}$, $x\in A_{e}$ and $I$ is an ideal of $A_{e}$.
\begin{enumerate}
    \item If $0\neq xIL\subseteq K$, then either $xL\subseteq K$ or $IL\subseteq K$.
    \item If $0\neq IxL\subseteq K$, then either $xL\subseteq K$ or $IL\subseteq K$.
\end{enumerate}
\end{cor}

\begin{thm}\label{Theorem 2.2 (1)}
Let $M$ be a graded $A$-module and $K$ be a $g$-classical weakly prime $A$-submodule of $M$. If $L$ is a faithful $A_{e}$-submodule of $M_{g}$ with $L\nsubseteq K_{g}$, then $(K_{g}:_{A_{e}}L)$ is a weakly prime left ideal of $A_{e}$. 
\end{thm}

\begin{proof}
Clearly, $(K_{g}:_{A_{e}}L)$ is a left ideal of $A_{e}$, and since $L\nsubseteq K_{g}$, $(K_{g}:_{A_{e}}L)$ is a proper ideal of $A_{e}$. Let $I, J$ be two ideals of $A_{e}$ such that $0\neq IJ\subseteq (K_{g}:_{A_{e}}L)$. Then $0\neq IJL\subseteq K_{g}\subseteq K$, and then by Theorem \ref{Theorem 2.1}, either $IL\subseteq K$ or $JL\subseteq K$. On the other hand, $IL\subseteq A_{e}M_{g}\subseteq M_{g}$, and similarly, $JL\subseteq M_{g}$. So, either $IL\subseteq K\bigcap M_{g}=K_{g}$ or $JL\subseteq K\bigcap M_{g}=K_{g}$, and thus either $I\subseteq (K_{g}:_{A_{e}}L)$ or $J\subseteq (K_{g}:_{A_{e}}L)$. Hence, $(K_{g}:_{A_{e}}L)$ is a weakly prime left ideal of $A_{e}$.  
\end{proof}

Similarly, one can prove the following:

\begin{thm}\label{Theorem 2.2 (1) (1)}
Let $M$ be a graded $A$-module and $K$ be a $g$-classical weakly prime $A$-submodule of $M$. If $L$ is a faithful $A_{e}$-submodule of $M_{g}$ with $L\nsubseteq K$, then $(K:_{A_{e}}L)$ is a weakly prime left ideal of $A_{e}$. 
\end{thm}

\begin{thm}\label{Theorem 2.2 (2)}
Let $M$ be a graded $A$-module and $K$ be a $g$-classical weakly prime $A$-submodule of $M$. If $Ann_{A_{e}}(M_{g})$ is a weakly prime ideal of $A_{e}$, then $(K_{g}:_{A_{e}}M_{g})$ is a weakly prime left ideal of $A_{e}$. 
\end{thm}

\begin{proof}
Since $K$ is a $g$-classical weakly prime $A$-submodule of $M$, $K_{g}\neq M_{g}$, and then $(K_{g}:_{A_{e}}M_{g})$ is a proper left ideal of $A_{e}$. Let $I, J$ be two ideals of $A_{e}$ such that $0\neq IJ\subseteq (K_{g}:_{A_{e}}M_{g})$. Then $IJM_{g}\subseteq K_{g}$. If $IJM_{g}\neq 0$, then $0\neq IJM_{g}\subseteq K_{g}\subseteq K$, and then by Theorem \ref{Theorem 2.1}, either $IM_{g}\subseteq K$ or $JM_{g}\subseteq K$. On the other hand, $IM_{g}\subseteq A_{e}M_{g}\subseteq M_{g}$, and similarly, $JM_{g}\subseteq M_{g}$. So, either $IM_{g}\subseteq K\bigcap M_{g}=K_{g}$ or $JM_{g}\subseteq K\bigcap M_{g}=K_{g}$, and hence either $I\subseteq (K_{g}:_{A_{e}}M_{g})$ or $J\subseteq (K_{g}:_{A_{e}}M_{g})$. If $IJM_{g}=0$, then $0\neq IJ\subseteq Ann_{A_{e}}(M_{g})$, and then either $I\subseteq Ann_{A_{e}}(M_{g})\subseteq (K_{g}:_{A_{e}}M_{g})$ or $J\subseteq Ann_{A_{e}}(M_{g})\subseteq (K_{g}:_{A_{e}}M_{g})$. Thus, $(K_{g}:_{A_{e}}M_{g})$ is a weakly prime left ideal of $A_{e}$. 
\end{proof}

Let $M, S$ be two $G$-graded $A$-modules. Then an $A$-homomorphism $f:M\rightarrow S$ is said to be a graded $A$-homomorphism if $f(M_{g})\subseteq S_{g}$ for all $g\in G$ \cite{Nastasescue}. 

\begin{thm}\label{Theorem 2.3}
Let $M, S$ be two $G$-graded $A$-modules and $f:M\rightarrow S$ be a graded $A$-homomorphism.
\begin{enumerate}
    \item If $f$ is injective and $K$ is a graded classical weakly prime $A$-submodule of $S$ with $f^{-1}(K)\neq M$, then $f^{-1}(K)$ is a graded classical weakly prime $A$-submodule of $M$.
    \item If $f$ is surjective and $K$ is a graded classical weakly prime $A$-submodule of $M$ with $Ker(f)\subseteq K$, then $f(K)$ is a graded classical weakly prime $A$-submodule of $S$. 
\end{enumerate}
\end{thm}

\begin{proof}
\begin{enumerate}
    \item By (\cite{Refai Dawwas}, Lemma 3.11 (1)), $f^{-1}(K)$ is a graded $A$-submodule of $M$. Let $x, y\in h(A)$ and $L$ be a graded $A$-submodule of $M$ such that $0\neq xAyL\subseteq f^{-1}(K)$. Then $f(L)$ is a graded $A$-submodule of $S$ by (\cite{Refai Dawwas}, Lemma 3.11 (2)) such that $0\neq xAyf(L)=f(xAyL)\subseteq K$, and then either $f(xL)=xf(L)\subseteq K$ or $f(yL)=yf(L)\subseteq K$, which implies that either $xL\subseteq f^{-1}(K)$ or $yL\subseteq f^{-1}(K)$. Hence, $f^{-1}(K)$ is a graded classical weakly prime $A$-submodule of $M$.
    \item By (\cite{Refai Dawwas}, Lemma 3.11 (2)), $f(K)$ is a graded $A$-submodule of $S$. Let $x, y\in h(A)$ and $L$ be a graded $A$-submodule of $S$ such that $0\neq xAyL\subseteq f(K)$. Then by (\cite{Refai Dawwas}, Lemma 3.11 (1)), $T=f^{-1}(L)$ is a graded $A$-submodule of $M$ such that $f(xAyT)=xAyf(T)=xAyL\subseteq f(K)$, and then $0\neq xAyT\subseteq K$ as $Ker(f)\subseteq K$. So, either $xT\subseteq K$ or $yT\subseteq K$, and then either $xL=xf(T)=f(xT)\subseteq f(K)$ or $yL=yf(T)=f(yT)\subseteq f(K)$. Hence, $f(K)$ is a graded classical weakly prime $A$-submodule of $S$. 
\end{enumerate}
\end{proof}

Let $M$ be a $G$-graded $A$-module and $T$ be a graded $A$-submodule of $M$. Then $M/T$ is a $G$-graded $A$-module by $(M/T)_{g}=(M_{g}+T)/T$ for all $g\in G$. By (\cite{Saber}, Lemma 3.2), if $K$ is an $A$-submodule of $M$ with $T\subseteq K$, then $K$ is a graded $A$-submodule of $M$ if and only if $K/T$ is a graded $A$-submodule of $M/T$. 

\begin{thm}\label{Corollary 2.4}
Let $M$ be a graded $A$-module and $T, K$ be a proper graded $A$-submodules of $M$ with $T\subsetneq K$. If $K$ is a graded classical weakly prime $A$-submodule of $M$, then $K/T$ is a graded classical weakly prime $A$-submodule of $M/T$. 
\end{thm}

\begin{proof}
Let $x, y\in h(A)$ and $L/T$ be a graded $A$-submodule of $M/T$ such that $(0+T)/T\neq (xAyL+T)/T=xAy((L+T)/T)\subseteq K/T$. Then $L$ is a graded $A$-submodule of $M$ such that $0\neq xAyL\subseteq K$, and then either $xL\subseteq K$ or $yL\subseteq K$, which implies that either $x((L+T)/T)=(xL+T)/T\subseteq (K+T)/T\subseteq K/T$ or $y((L+T)/T)=(yL+T)/T\subseteq (K+T)/T\subseteq K/T$. Hence, $K/T$ is a graded classical weakly prime $A$-submodule of $M/T$.  
\end{proof}

\begin{thm}\label{Theorem 2.4}
Let $M$ be a graded $A$-module and $T, K$ be a proper graded $A$-submodules of $M$ with $T\subsetneq K$. If $T$ is a graded classical weakly prime $A$-submodule of $M$ and $K/T$ is a graded classical weakly prime $A$-submodule of $M/T$, then $K$ is a graded classical weakly prime $A$-submodule of $M$. 
\end{thm}

\begin{proof}
Let $x, y\in h(A)$ and $L$ be a graded $A$-submodule of $M$ such that $xAyL\subseteq K$. If $xAyL\subseteq T$, then either $xL\subseteq T\subsetneq K$ or $yL\subseteq T\subsetneq K$ as required. Suppose that $xAyL\nsubseteq T$. Then $(0+T)/T\neq xAy((L+T)/T)=(xAyL+T)/T\subseteq (K+T)/T\subseteq K/T$, and then either $(xL+T)/T=x((L+T)/T)\subseteq K/T$ or $(yL+T)/T=y((L+T)/T)\subseteq K/T$, which implies that either $xL\subseteq K$ or $yL\subseteq K$. Hence, $K$ is a graded classical weakly prime $A$-submodule of $M$.
\end{proof}

Graded weakly $2$-absorbing ideals over non-commutative graded rings have been introduced and examined in \cite{Alshehry Habeb}. A proper graded ideal $P$ of $A$ is said to be a graded weakly $2$-absorbing ideal of $A$ if whenever $x, y, z\in h(A)$ such that $0\neq xAyAz\subseteq P$, then $xy\in P$ or $yz\in P$ or $xz\in P$. In this article, we introduce and investigate the concept of graded weakly $2$-absorbing submodules over non-commutative graded rings as follows:

\begin{defn} Let $M$ be a graded $A$-module and $K$ be a proper graded $A$-submodule of $M$. Then 
\begin{enumerate}
    \item $K$ is said to be a graded weakly $2$-absorbing submodule of $M$ if whenever $x, y\in h(A)$ and $L$ is a graded $A$-submodule of $M$ such that $0\neq xAyL\subseteq K$, then $xL\subseteq K$ or $yL\subseteq K$ or $xAy\subseteq (K:_{A}M)$.
    \item $K$ is said to be a graded completely weakly $2$-absorbing submodule of $M$ if whenever $x, y\in h(A)$ and $z\in h(M)$ such that $0\neq xyz\in K$, then $xz\in K$ or $yz\in K$ or $xy\in (K:_{A}M)$.
\end{enumerate}
\end{defn}

Clearly, every graded classical weakly prime submodule is graded weakly $2$-absorbing, and every graded completely classical weakly prime submodule is graded completely weakly $2$-absorbing. Also, if $A$ is a commutative graded ring with unity, then the concepts of graded weakly $2$-absorbing submodules and graded completely weakly $2$-absorbing submodules coincide. Indeed, this will not be the case for non-commutative graded rings; as in Example \ref{1}, we have $K$ is a graded weakly $2$-absorbing $A$-submodule of $M$ since it is a graded prime, but $K$ is not a graded completely weakly $2$-absorbing $A$-submodule of $M$ since $0\neq xyz\in K$, $xz\notin K$, $yz\notin K$ and $xy\nsubseteq (K:_{A}M)$.

\begin{prop}\label{Proposition 2.1}
Let $M$ be a graded $A$-module and $K$ be a proper graded $A$-submodule of $M$. If $K$ is a graded weakly $2$-absorbing $A$-submodule of $M$ and $(K:_{A}M)$ is a graded weakly prime ideal of $A$, then $K$ is a graded classical weakly prime $A$-submodule of $M$.
\end{prop}

\begin{proof}
Let $x, y\in h(A)$ and $L$ be a graded $A$-submodule of $M$ such that $0\neq xAyL\subseteq K$. Then $xL\subseteq K$ or $yL\subseteq K$ or $xAy\subseteq (K:_{A}M)$ since $K$ is a graded weakly $2$-absorbing $A$-submodule of $M$. If $xL\subseteq K$ or $yL\subseteq K$, then it is done. Suppose that $xAy\subseteq (K:_{A}M)$. Then as $0\neq xAy$, either $x\in (K:_{A}M)$ or $y\in (K:_{A}M)$ since $(K:_{A}M)$ is a graded weakly prime ideal of $A$, and then either $xL\subseteq xM\subseteq K$ or $yL\subseteq yM\subseteq K$. Hence, $K$ is a graded classical weakly prime $A$-submodule of $M$. 
\end{proof}

\begin{defn}
Let $M$ be a graded $A$-module and $K$ be a proper graded $A$-submodule of $M$. Then $K$ is said to be a graded classical prime $A$-submodule of $M$ if whenever $x, y\in h(A)$ and $L$ is a graded $A$-submodule of $M$ such that $xAyL\subseteq K$, then either $xL\subseteq K$ or $yL\subseteq K$.
\end{defn}

Clearly, every graded classical prime submodule is graded classical weakly prime. However, the next example shows that a graded classical weakly prime submodule is not necessarily graded classical prime:

\begin{exa}\label{11}
Consider $A=M=M_{2}(\mathbb{Z}_{8})$ and $G=\mathbb{Z}_{4}$. $M$ also is a $G$-graded left $A$-module by the same graduation of $A$. Then $A$ is $G$-graded by $A_{0}=\left(\begin{array}{cc}
        \mathbb{Z}_{8} & 0 \\
        0 & \mathbb{Z}_{8} 
      \end{array}\right)$, $A_{2}=\left(\begin{array}{cc}
        0 & \mathbb{Z}_{8} \\
        \mathbb{Z}_{8} & 0 
      \end{array}\right)$ and $A_{1}=A_{3}=\left(\begin{array}{cc}
        0 & 0 \\
        0 & 0 
      \end{array}\right)$. Now, $K=\left(\begin{array}{cc}
        0 & 0 \\
        0 & 0 
      \end{array}\right)$ is a graded classical weakly prime $A$-submodule of $M$, but $K$ is not a graded classical prime $A$-submodule of $M$ since $x=\left(\begin{array}{cc}
        2 & 0 \\
        0 & 2 
      \end{array}\right)\in h(A)$ and $L=Ax$ is a graded $A$-submodule of $M$ with $xAxL\subseteq K$ and $xL\nsubseteq K$.
\end{exa}

\begin{prop}\label{10}
Let $M$ be a graded $A$-module. If $K$ is a graded classical weakly prime $A$-submodule of $M$ which is not graded classical prime, then there exist $x, y\in h(A)$ and a graded $A$-submodule $L$ of $M$ such that $xAyL=0$, $xL\nsubseteq K$ and $yL\nsubseteq K$.
\end{prop}

\begin{proof}
Since $K$ is not graded classical prime $A$-submodule of $M$, there exist $x, y\in h(A)$ and a graded $A$-submodule $L$ of $M$ such that $xAyL\subseteq K$, $xL\nsubseteq K$ and $yL\nsubseteq K$. If $xAyL\neq 0$, then since $K$ is a graded classical weakly prime $A$-submodule of $M$, either $xL\nsubseteq K$ or $yL\nsubseteq K$, which is a contradiction. So, $xAyL=0$.
\end{proof}

\begin{defn}
Let $M$ be a graded $A$-module, $K$ be a proper graded $A$-submodule of $M$, $L$ be a graded $A$-submodule of $M$ and $x, y\in h(A)$. Then $(x, y, L)$ is said to be a graded classical triple zero of $K$ if $xAyL=0$, $xL\nsubseteq K$ and $yL\nsubseteq K$.
\end{defn}

\begin{rem}
If $K$ is a graded classical weakly prime $A$-submodule of $M$ which is not graded classical prime, then by Proposition \ref{10}, there exists a graded classical triple zero of $K$. Note that, in Example \ref{11}, $(x, x, L)$ is a graded classical triple zero of $K$.
\end{rem}

\begin{prop}\label{Proposition 2.2}
Let $M$ be a graded $A$-module, $K$ be a graded classical weakly prime $A$-submodule of $M$ and $xAyL\subseteq
 K$, for some $x, y\in h(A)$ and some graded $A$-submodule $L$ of $M$. If $(x, y, L)$ is not graded classical triple zero of $K$, then either $xL\subseteq K$ or $yL\subseteq K$.
\end{prop}

\begin{proof}
If $xAyL=0$, then since $(x, y, L)$ is not graded classical triple zero of $K$, either $xL\subseteq K$ or $yL\subseteq K$. If $xAyL\neq 0$, then since $K$ is a graded classical weakly prime $A$-submodule of $M$, either $xL\subseteq K$ or $yL\subseteq K$.
\end{proof}

\begin{cor}\label{Corollary 2.5}
Let $M$ be a graded $A$-module and $K$ be a graded classical weakly prime $A$-submodule of $M$. If $(x, y, L)$ is not graded classical triple zero of $K$, for all $x, y\in h(A)$ and all graded $A$-submodule $L$ of $M$, then $K$ is a graded classical prime $A$-submodule of $M$.
\end{cor}

\begin{prop}\label{Remark 2.1}
Let $M$ be a graded $A$-module, $K$ be a graded classical weakly prime $A$-submodule of $M$ and $IJL\subseteq K$, for some graded ideals $I, J$ of $A$ and some graded $A$-submodule $L$ of $M$. If $(x, y, L)$ is not graded classical triple zero of $K$, for all $x\in I\bigcap h(A)$ and $y\in J\bigcap h(A)$, then for all $x\in I$, $y\in J$ and $g\in G$, we have either $x_{g}L\subseteq K$ or $y_{g}L\subseteq K$. 
\end{prop}

\begin{proof}
Let $x\in I$, $y\in J$ and $g\in G$. Then $x_{g}\in I$ and $y_{g}\in J$ since $I$ and $J$ are graded ideals, and then $x_{g}Ay_{g}L\subseteq IJL\subseteq K$. If $x_{g}Ay_{g}L\neq 0$, then since $K$ is a graded classical weakly prime $A$-submodule of $M$, either $x_{g}L\subseteq K$ or $y_{g}L\subseteq K$. If $x_{g}Ay_{g}L=0$, then since $(x_{g}, y_{g}, L)$ is not graded classical triple zero of $K$, either $x_{g}L\subseteq K$ or $y_{g}L\subseteq K$.
\end{proof}

\begin{cor}\label{Proposition 2.3}
Let $M$ be a graded $A$-module, $K$ be a graded classical weakly prime $A$-submodule of $M$ and $IJL\subseteq K$, for some graded ideals $I, J$ of $A$ and some graded $A$-submodule $L$ of $M$. If $(x, y, L)$ is not graded classical triple zero of $K$, for all $x\in I\bigcap h(A)$ and $y\in J\bigcap h(A)$, then for all $g\in G$, we have either $I_{g}L\subseteq K$ or $J_{g}L\subseteq K$.
\end{cor}

\begin{thm}\label{Theorem 2.5}
Let $M$ be a graded $A$-module and $K$ be a graded classical weakly prime $A$-submodule of $M$. Let $L$ be a graded $A$-submodule of $M$ and $x, y\in h(A)$. If $(x, y, L)$ is a graded classical triple zero of $K$, then the following statements hold:
\begin{enumerate}
    \item $xAyK=0$.
    \item If $y\in A_{g}$, for some $g\in G$, then $x(K:_{A_{g}}M)L=0$.
    \item If $x\in A_{g}$, for some $g\in G$, then $(K:_{A_{g}}M)yL=0$.
    \item If $x, y\in A_{g}$, for some $g\in G$, then $(K:_{A_{g}}M)^{2}L=0$.
    \item If $y\in A_{g}$, for some $g\in G$, then $x(K:_{A_{g}}M)K=0$.
    \item If $x\in A_{g}$, for some $g\in G$, then $(K:_{A_{g}}M)yK=0$.
    \item If $x, y\in A_{g}$, for some $g\in G$, then $(K:_{A_{g}}M)^{2}K=0$.
\end{enumerate}
\end{thm}

\begin{proof}
\begin{enumerate}
    \item Suppose that $xAyK\neq 0$. Then there exists $z\in K$ such that $xAyz\neq 0$, and then there exists $g\in G$ such that $xAyz_{g}\neq 0$, and hence $xAyAz_{g}\neq 0$. Note that $z_{g}\in K$ as $K$ is a graded $A$-submodule. Let $T=L+Az_{g}$. Then $T$ is a graded $A$-submodule of $M$ such that $0\neq xAyT\subseteq K$, and then either $xT\subseteq K$ or $yT\subseteq K$, which implies that either $xL\subseteq xT\subseteq K$ or $yL\subseteq yT\subseteq K$, which is a contradiction. Thus $xAyK=0$. 
    \item Suppose that $x(K:_{A_{g}}M)L\neq 0$. Then there exists $s\in (K:_{A_{g}}M)$ such that $xsL\neq 0$, and then $xAsL\neq 0$, and hence $0\neq xA(y+s)L\subseteq K$. So, either $xL\subseteq K$ or $(y+s)L\subseteq K$, and then either $xL\subseteq K$ or $yL\subseteq K$, which is a contradiction. Hence, $x(K:_{A_{g}}M)L=0$.
    \item Similar to (2).
    \item Suppose that $(K:_{A_{g}}M)^{2}L\neq 0$. Then there exist $r, s\in (K:_{A_{g}}M)$ such that $rsL\neq 0$, and then $rAsL\neq 0$, and hence by (2) and (3), $0\neq (x+r)A(y+s)L\subseteq K$. So, either $(x+r)L\subseteq K$ or $(y+s)L\subseteq K$, and then either $xL\subseteq K$ or $yL\subseteq K$, which is a contradiction. Hence, $(K:_{A_{g}}M)^{2}L=0$.
    \item Suppose that $x(K:_{A_{g}}M)K\neq 0$. Then there exists $s\in (K:_{A_{g}}M)$ such that $xsK\neq 0$, and then $xAsK\neq 0$, and hence by (1), $xA(y+s)K\neq 0$, which implies that $xA(y+s)z\neq 0$, for some $z\in K$, and so there exists $h\in G$ such that $xA(y+s)z_{h}\neq 0$. Note that $z_{h}\in K$ as $K$ is a graded $A$-submodule. Let $T=L+Az_{h}$. Then $T$ is a graded $A$-submodule of $M$ such that $0\neq xA(y+s)T\subseteq K$, and then either $xT\subseteq K$ or $(y+s)T\subseteq K$, and hence either $xT\subseteq K$ or $yT\subseteq K$, which implies that either $xL\subseteq xT\subseteq K$ or $yL\subseteq yT\subseteq K$, which is a contradiction. Hence, $x(K:_{A_{g}}M)K=0$.
    \item Similar to (5).
    \item Suppose that $(K:_{A_{g}}M)^{2}K\neq 0$. Then there exist $r, s\in (K:_{A_{g}}M)$ and $z\in K$ such that $rsz\neq 0$, and then there exists $h\in G$ such that $rsz_{h}\neq 0$. Note that $z_{h}\in K$ as $K$ is a graded $A$-submodule. Let $T=L+Az_{h}$. Then $T$ is a graded $A$-submodule of $M$ such that $0\neq (x+r)A(y+s)T\subseteq K$ by (2) and (3), and then either $(x+r)T\subseteq K$ or $(y+s)T\subseteq K$, and hence either $xT\subseteq K$ or $yT\subseteq K$, which implies that either $xL\subseteq xT\subseteq K$ or $yL\subseteq yT\subseteq K$, which is a contradiction. Hence, $(K:_{A_{g}}M)^{2}K=0$.
\end{enumerate}
\end{proof}

\begin{prop}\label{Proposition 3.1 (1)}
Let $M$ be a graded $A$-module and $K$ be a graded classical weakly prime $A$-submodule of $M$. If there exists a graded classical triple zero $(x, y, L)$ of $K$ with $x, y\in A_{e}$, then $(K:_{A_{e}}M)^{3}\subseteq Ann_{A_{e}}(M)$.
\end{prop}

\begin{proof}
    By Theorem \ref{Theorem 2.5} (7), $(K:_{A_{e}}M)^{2}K=0$. So, $(K:_{A_{e}}M)^{3}=(K:_{A_{e}}M)^{2}(K:_{A_{e}}M)\subseteq ((K:_{A_{e}}M)^{2}K:_{A_{e}}M)=(0:_{A_{e}}M)=Ann_{A_{e}}(M)$.
\end{proof}

As a consequence of Proposition \ref{Proposition 3.1 (1)}, we have the following:

\begin{cor}
Let $M$ be a faithful graded $A$-module and $K$ be a graded classical weakly prime $A$-submodule of $M$. If there exists a graded classical triple zero $(x, y, L)$ of $K$ with $x, y\in A_{e}$, then $(K:_{A_{e}}M)^{3}=0$.
\end{cor}

Also, as a consequence of Theorem \ref{Theorem 2.5} (7), we have the following:

\begin{cor}
Let $M$ be a graded $A$-module and $K$ be a faithful graded classical weakly prime $A$-submodule of $M$. If there exists a graded classical triple zero $(x, y, L)$ of $K$ with $x, y\in A_{g}$, for some $g\in G$, then $(K:_{A_{g}}M)^{2}=0$.
\end{cor}

Furthermore, as a consequence of Theorem \ref{Theorem 2.5} (4), we have the following:

\begin{cor}
Let $M$ be a graded $A$-module and $K$ be a graded classical weakly prime $A$-submodule of $M$. If there exists a graded classical triple zero $(x, y, L)$ of $K$ with $x, y\in A_{g}$, for some $g\in G$, and $L$ is faithful, then $(K:_{A_{g}}M)^{2}=0$.
\end{cor}

Let $M$ and $S$ be two $G$-graded $A$-modules. Then $M\times S$ is a $G$-graded $A$-module by $(M\times S)_{g}=M_{g}\times S_{g}$, for all $g\in G$ \cite{Nastasescue}. Moreover, $N=K\times T$ is a graded $A$-submodule of $M\times S$ if and only if $K$ is a graded $A$-submodule of $M$ and $T$ is a graded $A$-submodule of $S$ (\cite{Saber}, Lemma 3.10 and Lemma 3.12).

\begin{thm}\label{Theorem 2.7 (1)}
Let $M$ and $S$ be two $G$-graded $A$-modules. If $K\times S$ is a graded classical weakly prime $A$-submodule of $M\times S$, then $K$ is a graded classical weakly prime $A$-submodule of $M$.
\end{thm}

\begin{proof}
Let $x, y\in h(A)$ and $L$ be a graded $A$-submodule of $M$ such that $0\neq xAyL\subseteq K$. Then $L\times \{0\}$ is a graded $A$-submodule of $M\times S$ such that $(0, 0)\neq xAy(L\times \{0\})\subseteq K\times S$, and then either $x(L\times \{0\})\subseteq K\times S$ or $y(L\times \{0\})\subseteq K\times S$, and hence either $xL\subseteq K$ or $yL\subseteq K$. Thus $K$ is a graded classical weakly prime $A$-submodule of $M$.  
\end{proof}

\begin{thm}\label{Theorem 2.7 (2)}
Let $M$ and $S$ be two $G$-graded $A$-modules and $K\times S$ is a graded classical weakly prime $A$-submodule of $M\times S$. If $(x, y, L)$ is a graded classical triple zero of $K$, then $xAy\subseteq Ann_{A}(S)$.
\end{thm}

\begin{proof}
Suppose that $xAy\nsubseteq Ann_{A}(S)$. Then there exists $s\in S$ such that $xAys\neq 0$, and then there exists $g\in G$ such that $xAys_{g}\neq 0$. Now, $L\times As_{g}$ is a graded $A$-submodule of $M\times S$ such that $(0, 0)\neq xAy(L\times As_{g})\subseteq K\times S$, so either $x(L\times As_{g})\subseteq K\times S$ or $y(L\times As_{g})\subseteq K\times S$, and hence either $xL\subseteq K$ or $yL\subseteq K$, which is a contradiction. Thus $xAy\subseteq Ann_{A}(S)$.
\end{proof}

We close this section by introducing a nice result concerning graded weakly prime submodules over graded multiplication modules (Theorem \ref{51}). A graded $A$-module $M$ is said to be graded multiplication if for every graded $A$-submodule $K$ of $M$, $K=IM$, for some graded deal $I$ of $A$. In this case, it is known that $I=(K:_{A}M)$. Graded multiplication modules were first introduced and studied by Escoriza and Torrecillas in \cite{Escoriza}, and further results were obtained by several authors, see for example \cite{Bataineh Shtayat, Khaksari}.

\begin{lem}\label{50}
Let $M$ be a graded $A$-module. Then every graded maximal $A$-submodule of $M$ is graded prime. 
\end{lem}

\begin{proof}
    Let $K$ be a graded maximal $A$-submodule of $M$. Suppose that $IL\subseteq K$, for some graded ideal $I$ of $A$ and some graded $A$-submodule $L$ of $M$. Assume that $L\nsubseteq K$. Since $K$ is a graded maximal $A$-submodule of $M$, $L+K=M$, and hence $IM=IL+IK\subseteq K$, which implies that $I\subseteq (K:_{A}M)$. Hence, $K$ is a graded prime $A$-submodule of $M$.
\end{proof}

\begin{thm}\label{51}
    Let $M$ be a graded multiplication $A$-module. If every proper graded $A$-submodule of $M$ is graded weakly prime, then $M$ has at most two graded maximal $A$-submodules. 
\end{thm}

\begin{proof}
    Let $X$, $Y$ and $Z$ be three distinct graded maximal $A$-submodules of $M$. Since $M$ is graded multiplication, $X=IM$, for some graded ideal $I$ of $A$. If $IY=0$, then $IY\subseteq Z$, and since $Z$ is graded prime by Lemma \ref{50}, either $Y\subseteq Z$ or $X=IM\subseteq Z$, which is a contradiction. So, $IY\neq 0$. Now, $IY\subseteq Y$ and $IY\subseteq IM=X$, so $0\neq IY\subseteq X\bigcap Y$, and since $X\bigcap Y$ is graded weakly prime by assumption, either $Y\subseteq X\bigcap Y$ or $X=IM\subseteq X\bigcap Y$, which implies that either $Y\subseteq X$ or $X\subseteq Y$, which is a contradiction. Hence, $M$ has at most two graded maximal $A$-submodules.
\end{proof}

\begin{cor}
  Let $M$ be a graded multiplication $A$-module such that every proper graded $A$-submodule of $M$ is graded weakly prime. If $X=IM$ and $Y=JM$ are two distinct graded $A$-submodules of $M$, for some graded ideals $I, J$ of $A$, then either $X$ and $Y$ are comparable by inclusion or $IY=JX=0$. In particular, if $X$ and $Y$ are two distinct graded maximal $A$-submodules of $M$, then $IY=JX=0$.  
\end{cor}

\section{Graded Classical Weakly Prime Submodules over Duo Graded Rings}

In this section, we study graded classical weakly prime submodules over Duo graded rings. A ring $A$ is said to be a left Duo ring if every left ideal of $A$ is a two sided ideal \cite{Habeb, Marks}. It is obvious that if $A$ is a left Duo ring, then $xA\subseteq Ax$, for all $x\in A$.

\begin{thm}\label{Theorem 3.1}
Let $A$ be a left Duo ring, $M$ be a graded $A$-module and $K$ be a graded classical weakly prime $A$-submodule of $M$. If $x, y\in h(A)$ and $m\in h(M)$ such that $0\neq xym\in K$, then either $xm\in K$ or $ym\in K$.
\end{thm}

\begin{proof}
Since $A$ is a left Duo ring, $Axy=AxAyA$, and then $Am$ is a graded $A$-submodule of $M$ such that $0\neq xAyAm\subseteq K$, and so either $xAm\subseteq K$ or $yAm\subseteq K$, that implies either $xm\in K$ or $ym\in K$. 
\end{proof}

Theorem \ref{Theorem 3.1} says that every graded classical weakly prime submodule of a graded module  over  a left duo ring is graded completely classical weakly prime. As a consequence of Theorem \ref{Theorem 3.1}, we have the following:

\begin{cor}\label{Corollary 3.1}
Let $A$ be a left Duo ring, $M$ be a graded $A$-module and $K$ be a graded classical weakly prime $A$-submodule of $M$. If $x, y\in h(A)$ and $m\in h(M)$ such that $xym\in K$ and $(x, y, Am)$ is not graded classical triple zero of $K$, then either $xm\in K$ or $ym\in K$.
\end{cor}

Let $M$ be an $A$-module and $K$ be an $A$-submodule of $M$. Then $K$ is said to be an $u$-submodule if whenever $K\subseteq \displaystyle\bigcup_{j=1}^{n}K_{j}$, for some $A$-submodules $K_{j}$'s of $M$, then $K\subseteq K_{j}$, for some $1\leq j\leq n$. $M$ is said to be an $u$-module if every $A$-submodule of $M$ is an $u$-submodule \cite{PHILIP}. Let $M$ be an $A$-module, $K$ be an $A$-submodule of $M$ and $r\in A$. Then $(K:_{M}r)=\{m\in M: rm\in K\}$ is an $A$-submodule of $M$ containing $K$.

\begin{thm}\label{Theorem 3.2}
Let $M$ be a graded $A$-module such that $A_{e}$ is a left Duo ring and $M_{g}$ is an $u$-module over $A_{e}$, for some $g\in G$. Suppose that $K$ is a graded $A$-submodule of $M$ such that $K_{g}\neq M_{g}$. Consider the following statements:
\begin{enumerate}
    \item $K$ is a $g$-classical weakly prime $A$-submodule of $M$.
    \item If $x, y\in A_{e}$ and $m\in M_{g}$ such that $0\neq xym\in K$, then either $xm\in K$ or $ym\in K$.
    \item For all $x, y\in A_{e}$, $(K:_{M_{g}}xy)=(0:_{M_{g}}xy)$ or $(K:_{M_{g}}xy)=(K:_{M_{g}}x)$ or $(K:_{M_{g}}xy)=(K:_{M_{g}}y)$.
    \item If $x, y\in A_{e}$ and $L$ is an $A_{e}$-submodule of $M_{g}$ such that $0\neq xyL\subseteq K$, then either $xL\subseteq K$ or $yL\subseteq K$.
    \item If $x\in A_{e}$ and $L$ is an $A_{e}$-submodule of $M_{g}$ such that $xL\nsubseteq K$, then either $(K:_{A_{e}}xL)=(0:_{A_{e}}xL)$ or $(K:_{A_{e}}xL)=(K:_{A_{e}}L)$.
    \item If $x\in A_{e}$, $I$ is an ideal of $A_{e}$ and $L$ is an $A_{e}$-submodule of $M_{g}$ such that $0\neq IxL\subseteq K$, then either $IL\subseteq K$ or $xL\subseteq K$.
    \item If $I$ is an ideal of $A_{e}$ and $L$ is an $A_{e}$-submodule of $M_{g}$ such that $IL\nsubseteq K$, then either $(K:_{A_{e}}IL)=(0:_{A_{e}}IL)$ or $(K:_{A_{e}}IL)=(K:_{A_{e}}L)$.
\end{enumerate}
Then $(1)\Rightarrow (2)\Rightarrow (3)\Rightarrow (4)\Rightarrow (5)\Rightarrow (6)\Rightarrow (7)$.
\end{thm}

\begin{proof}
$(1)\Rightarrow (2)$: Similar to proof of Theorem \ref{Theorem 3.1}.

$(2)\Rightarrow (3)$: Let $x, y\in A_{e}$ and $m\in (K:_{M_{g}}xy)$. Then $xym\in K$. If $xym=0$, then $m\in (0:_{M_{g}}xy)$. If $xym\neq 0$, then by (2), either $xm\in K$ or $ym\in K$, and so either $m\in (K:_{M_{g}}x)$ or $m\in (K:_{M_{g}}y)$. Hence, $(K:_{M_{g}}xy)\subseteq (0:_{M_{g}}xy)\bigcup (K:_{M_{g}}x)\bigcup(K:_{M_{g}}y)$, and then $(K:_{M_{g}}xy)\subseteq (0:_{M_{g}}xy)$ or $(K:_{M_{g}}xy)\subseteq (K:_{M_{g}}x)$ or $(K:_{M_{g}}xy)\subseteq (K:_{M_{g}}y)$ as $M_{g}$ is an $u$-module over $A_{e}$, and hence the result holds as $A_{e}$ is a left Duo ring.

$(3)\Rightarrow (4)$: Clearly, $L\subseteq (K:_{M_{g}}xy)$ and $K\nsubseteq (0:_{M_{g}}xy)$, so by (3), either $L\subseteq (K:_{M_{g}}x)$ or $L\subseteq (K:_{M_{g}}y)$, and then either $xL\subseteq K$ or $yL\subseteq K$.

$(4)\Rightarrow (5)$: Let $r\in (K:_{A_{e}}xL)$. Then $rxL\subseteq K$. If $rxL=0$, then $r\in (0:_{A_{e}}xL)$. If $rxL\neq 0$, then by (4), $rL\subseteq K$. So, $(K:_{A_{e}}xL)\subseteq (0:_{A_{e}}xL)\bigcup (K:_{A_{e}}L)$. Hence, the result holds as $A_{e}$ is a left Duo ring.

$(5)\Rightarrow (6)$: Clearly, $I\subseteq (K:_{A_{e}}xL)$ and $I\nsubseteq (0:_{A_{e}}xL)$, so by (5), either $xL\subseteq K$ or $I\subseteq (K:_{A_{e}}L)$, and then either $xL\subseteq K$ or $IL\subseteq K$.

$(6)\Rightarrow (7)$: Let $x\in (K:_{A_{e}}IL)$. Then $xIL\subseteq K$. If $xIL=0$, then $x\in (0:_{A_{e}}IL)$. If $xIL\neq 0$, then as $A_{e}$ is a left Duo ring, $0\neq IxL\subseteq K$, and then by (6), $xL\subseteq K$, and hence $x\in (K:_{A_{e}}L)$. Thus $(K:_{A_{e}}IL)\subseteq (0:_{A_{e}}IL)\bigcup (K:_{A_{e}}L)$. So, the result holds as $A_{e}$ is a left Duo ring.
\end{proof}

Let $A$ be a left Duo ring. If $A$ is a graded ring, then clearly, $Grad_{A}(I)=\{x\in A:\forall g\in G,\exists n_{g}\in\mathbb{N}\text{ s.t. }x_{g}^{n_{g}}\in I\}$ is a graded ideal of $A$ containing $I$. Evidently, if $x\in h(A)$, then $x\in Grad_{A}(I)$ if and only if $x^{n}\in I$ for some $n\in \mathbb{N}$.

\begin{prop}\label{Proposition 3.1 (2)}
Let $M$ be a graded $A$-module such that $A_{e}$ is a left Duo ring and $K$ be a graded classical weakly prime $A$-submodule of $M$. If there exists a graded classical triple zero $(x, y, L)$ of $K$ with $x, y\in A_{e}$, then $Grad_{A_{e}}(Ann_{A_{e}}(M))=Grad_{A_{e}}((K:_{A_{e}}M))$.
\end{prop}

\begin{proof}
    By Proposition \ref{Proposition 3.1 (1)}, $(K:_{A_{e}}M)\subseteq Grad_{A_{e}}(Ann_{A_{e}}(M))$, and then $Grad_{A_{e}}((K:_{A_{e}}M))\subseteq Grad_{A_{e}}(Grad_{A_{e}}(Ann_{A_{e}}(M)))=Grad_{A_{e}}(Ann_{A_{e}}(M))$. On the other hand, since $Ann_{A_{e}}(M)\subseteq (K:_{A_{e}}M)$, $Grad_{A_{e}}(Ann_{A_{e}}(M))\subseteq Grad_{A_{e}}((K:_{A_{e}}M))$. Hence, $Grad_{A_{e}}(Ann_{A_{e}}(M))=Grad_{A_{e}}((K:_{A_{e}}M))$.
\end{proof}

\end{document}